\documentclass[12pt, reqno]{amsart}
\usepackage[bookmarksnumbered,colorlinks, plainpages]{hyperref}
\hypersetup{colorlinks=true,linkcolor=red, anchorcolor=green, citecolor=cyan, urlcolor=red, filecolor=magenta, pdftoolbar=true}

\usepackage{subfig}
\usepackage{rotating}
\usepackage{tikz}

\usepackage{fixltx2e}
\usepackage{array}
\RequirePackage{tkz-berge}

\usepackage{amsmath, amsthm, amscd, amsfonts, amssymb, graphicx, color, enumerate, xcolor,  mathrsfs, latexsym}

 \definecolor{cupgreen}{rgb}{0,0.498,0.208}
  \definecolor{cupblue}{rgb}{0,0,.5}
  \definecolor{cupred}{rgb}{1,0.04,0}
  \definecolor{cuppink}{rgb}{0.925,0,0.545}
  \definecolor{cupmagenta}{rgb}{0.624,0.161,0.424}
  \definecolor{cupbrown}{rgb}{0.71,0.212,0.133}

  \definecolor{cupgreen}{rgb}{0,0,0}
  \definecolor{cupblue}{rgb}{0,0,0}
  \definecolor{cupred}{rgb}{0,0,0}
  \definecolor{cuppink}{rgb}{0,0,0}
  \definecolor{cupmagenta}{rgb}{0,0,0}
  \definecolor{cupbrown}{rgb}{0,0,0}

\definecolor{TITLE}{rgb}{0,0,0}
\definecolor{AUTHOR1}{rgb}{0.00,0.59,0.00}
\definecolor{AUTHOR2}{rgb}{0.50,0.00,1.00}
\definecolor{SECTION}{rgb}{0.50,0.00,1.00}
\definecolor{FOOTTITLE}{rgb}{0.00,0.50,0.75}
\definecolor{THM}{rgb}{0.8,0,0.1}
\definecolor{SEC}{rgb}{0,0,1}

\makeatletter

\newtheorem{theorem}{{\color{THM} Theorem}}[section]
\title[Partitions with parts in a finite set]{Partitions with parts in a finite set and the non-intersecting circles problem}

\newtheorem{lemma}[theorem]{{\color{THM}Lemma}}

\theoremstyle{definition}

\newtheorem{conjecture}[theorem]{{\color{THM}Conjecture\ }}

\numberwithin{equation}{section}

\usepackage{amscd}
\begin{document}

\newbox\Adr
\setbox\Adr\vbox{
\centerline{\sc Daniel Yaqubi$^1$ and Madjid Mirzavaziri$^2$}
\vskip18pt
\centerline{$^{1,2}$Department of Pure Mathematics, Ferdowsi University of Mashhad}
\centerline{P. O. Box 1159, Mashhad 91775, Iran.}
\centerline{Email$^{1}$: {\tt daniel\_yaqubi@yahoo.es } }
\centerline{Email$^{2}$: {\tt mirzavaziri@um.ac.ir}}}

\title{Some divisibility properties of binomial coefficients}

\author[ Daniel Yaqubi and Madjid Mirzavaziri ]{\box\Adr}

\keywords{Binomial coefficients, $p$-adic valuation, Lucas' theorem, Euler's totient theorem, Bernoulli numbers.}

\subjclass[2010]{Primary 11B65; Secondary 05A10}

\maketitle
\begin{abstract}
In this paper, we aim to give full proofs or partial answers for the following three conjectures of V. J. W. Guo and C. Krattenthaler: (1) Let $a>b$ be positive integers, $\alpha,\beta$ be any integers and $p$ be a prime satisfying $\gcd(p,a)=1$. Then there exist infinitely many positive integers $n$ for which $\binom{an+\alpha}{bn+\beta}\equiv r\pmod p$ for all integers $r$; (2) For any odd prime $p$, there are no positive integers $a>b$ such that $\binom{an}{bn}\equiv0\pmod{pn-1}$ for all $n\geq1$; (3) For any positive integer $m$, there exist positive integers $a$ and $b$ such that $am>b$ and $\binom{amn}{bn}\equiv0\pmod{an-1}$ for all $n\geq1$. Moreover, we show that for any positive integer $m$, there are positive integers $a$ and $b$ such that 
${amn \choose bn}\equiv 0\quad (\mbox{\rm mod~} an-a)$ for all $n\geqslant 1$.
\end{abstract}
\section{introduction}

Binomial coefficients constitute an important class of numbers that arise naturally in mathematics, namely as coefficients in the expansion of the polynomial $(x+y)^n$. Accordingly, they appear in various mathematical areas. An elementary property of binomial coefficients is that $\binom{n}{m}$ is divisible by a prime $p$ for all $1<m<n$ if and only if $n$ is a power of $p$. A much more technical result is due to Lucas, which asserts that
\[\binom{n}{m}\equiv\binom{n_0}{m_0}\binom{n_1}{m_1}\cdots\binom{n_k}{m_k}\pmod p,\]
in which $n=n_0+n_1p+\cdots+n_kp^k$ and $m=m_0+m_1p+\cdots+m_kp^k$  the $p$-adic expansions of the non-negative integers $n$ and $m$, respectively. We note that $0\leq m_i,n_i<p$, for all $i=0,\ldots,k$. In 1819, Babbage  \cite{Bab} revealed the following congruences for all odd prime $p$:
\[\binom{2p-1}{p-1}\equiv1\pmod{p^2}.\]
In 1862, \textit{ Wolstenholme} \cite{Wolstenholme} strengthened the identity of Babbage by showing that the same congruence holds modulo $p^3$ for all prime $p\geq5$. This identity was further generalized by Ljunggren in 1952 to $\binom{np}{mp}\equiv\binom{n}{m}\pmod{p^3}$ and even more to $\binom{np}{mp}/\binom{n}{m}\equiv1\pmod{p^q}$ by Jacobsthal for all positive integers $n>m$ and primes $p\geq5$, in which $p^q$ is any power of $p$ dividing $p^3mn(n-m)$. Note that the number $q$ can be replaced by a large number if $p$ divides $B_{p-1}$, the $(p-3)$'th Bernoulli number. Arithmetic properties of binomial coefficients are studied extensively in the literature and we may refer the interested reader to \cite{3} for an account of Wolstenholme's theorem. Recently,  \textit{Guo} and \textit{Krattenthaler} \cite{KRA} studied a similar problem and proved the following conjecture of \textit{Sun} \cite{Sun3}.
\begin{theorem}
Let $a$ and $b$ be positive integers. If $bn+1$ divides $\binom{an+bn}{an}$ for all sufficiently large positive integers $n$, then each prime factor of $a$ divides $b$. In other words, if $a$ has a prime factor not dividing $b$, then there are infinitely many positive integers $n$ for which $bn+1$ does not divide $\binom{an+bn}{an}$.
\end{theorem}

They also stated several conjectures among which are the following, which we
aim to give full proofs for two conjectures and a partial answer for one of them.
\begin{conjecture}[{\cite[Conjecture 7.1]{KRA}}]\label{Con1}
Let $a>b$ be positive integers, $\alpha,\beta$ be any integers and $p$ be a prime satisfying $\gcd(p,a)=1$. Then there exist infinitely many positive integers $n$ for which 
\[\binom{an+\alpha}{bn+\beta}\equiv r\pmod p\]
for all integers $r$. 
\end{conjecture}
\begin{conjecture}[{\cite[Conjecture 7.2]{KRA}}]\label{Con2}
For any odd prime $p$, there are no positive integers $a>b$ such that 
\[\binom{an}{bn}\equiv0\pmod{pn-1}\]
for all $n\geq1$.
\end{conjecture}
\begin{conjecture}[{\cite[Conjecture 7.3]{KRA}}]\label{Con3}
For any positive integer $m$, there exist positive integers $a$ and $b$ such that $am>b$ and 
\[\binom{amn}{bn}\equiv0\pmod{an-1}\]
for all $n\geq1$. 
\end{conjecture}

Moreover, we show that for any positive integer $m$, there are positive integers $a$ and $b$ such that 
${amn \choose bn}\equiv 0\quad (\mbox{\rm mod~} an-a)$ for all $n\geqslant 1$.

\section{Conjecture  \ref{Con2}} 
Our first result is a more precise version of Conjecture \ref{Con2} in this case that $a\not\equiv 0~  (\mbox{\rm mod~} p)$ and we obtain some divisibility property of binomial coefficients. 
\begin{theorem}
For any odd prime $p$, there are no positive integers $a>b$ with $a\not\equiv 0~  (\mbox{\rm mod~} p)$ such that
\[{an \choose bn}\equiv 0\quad (\mbox{\rm mod~} pn-1),\]
for all $n\geqslant 1$.
\end{theorem}
\begin{proof}
There are two cases.

\noindent\textbf{Case I}. $a\not\equiv 0~  (\mbox{\rm mod~} p)$ and $b\not\equiv 0~  (\mbox{\rm mod~} p)$.

There is an $1\leqslant s\leqslant p-1$ such that $sb\equiv 1~ (\mbox{\rm mod~} p)$. We may write
\begin{eqnarray*}
sa&=& pQ+r,\quad 1\leqslant r\leqslant p-1\\
sb&=&pQ'+1.
\end{eqnarray*} 
Choose $t$ such that $st\equiv -1~(\mbox{\rm mod~} p)$. Thus there is a $k$ with $st=kp-1$. We claim that
\[(pn+t)(p+s)=pK-1\mid{aK\choose bK},\]
where $K=pn+ns+t+k$. By Dirichlet's theorem, there are infinitely many primes of the form $pn+t$. If $pn+t$ is prime, Lucas' theorem implies that
\begin{eqnarray*}
{aK\choose bK}&=&{a(pn+ns+t+k)\choose b(pn+ns+t+k)}\\
&=&{a(pn+t)+a(ns+k)\choose b(pn+t)+b(ns+k)}\\
&=&{a(pn+t)+Q(pn+t)+rn+ak-Qt\choose b(pn+t)+Q'(pn+t)+n+bk-Q't}\\
&\equiv&{a+Q\choose b+Q'}{rn+ak-Qt\choose n+bk-Q't}\quad(\mbox{\rm mod~}pn+t),
\end{eqnarray*}
since for sufficiently large $n$ we have $rn+ak-Qt,n+bk-Q't<pn+t$.

Now we have 
\begin{eqnarray*}
s(n+bk-Q't)&=&sn+(pQ'+1)k-Q'(pk-1)\\
&=&sn+k+Q'\\
&\leqslant& srn+rk+Q\\
&\leqslant& srn+(pQ+r)k-Q(pk-1)=s(rn+ak-Qt).
\end{eqnarray*}
Whence ${rn+ak-Qt\choose n+bk-Q't}\neq 0$.

\noindent\textbf{Case II}. $a\not\equiv 0~  (\mbox{\rm mod~} p)$ and $b\equiv 0~  (\mbox{\rm mod~} p)$.

We should have 
\[0\equiv {aK\choose bK}\equiv{a+Q\choose b+Q'}{rn+ak-Qt\choose bk-Q't}\quad(\mbox{\rm mod~}pn+t),\]
where $K=pn+t+k$. That is impossible for sufficiently large $n$.
\end{proof}
In the next theorem we aim at considering the case $a=cp$ and $b=pk+r$, where $1\leqslant r\leqslant p-1$ and we give a partial answer for Conjecture  \ref{Con2} in this case.

We know that for each prime number $p$ and $\varepsilon>0$ there is a real number $M_p(\varepsilon)$ such that for each $x\geqslant M_p(\varepsilon)$ there is a prime number $q$ in the interval $(x,(1+\varepsilon) x)$ with $q\equiv -1~(\mbox{\rm mod~} p)$ \cite{Har}. 
Moreover, there is a real number $M'_p(\varepsilon)$ such that for each $x\geqslant M'_p(\varepsilon)$ there are at least two prime numbers $q,q'$ in the interval $(x,(1+\varepsilon) x)$ with $q,q'\equiv -1~(\mbox{\rm mod~} p)$. 

In the following we may assume $b< c(p-r)$, since if $b\geqslant c(p-r)$ then ${pcn \choose bn}={pcn\choose b'n}$, where $b'=pc-b$. We have $b'=pk'+r'$, where $k'=c-k-1, r'=p-r$ and $b'<c(p-r')$. 
\begin{theorem}\label{main}
Let $p$ be an odd prime, $1\leqslant r\leqslant p-1$ and $\gamma=\frac{p^2r+p^2+r^2-pr}{p^2(p+1)}$. 
\begin{itemize}
\item[i.] If $pk+r\leqslant cp(1-\gamma)$ then there are no positive integers $c\geqslant M_p(\frac{(p-r)^2}{pr(p+1)})$ and $k$ such that
\[{pcn\choose (pk+r)n}\equiv 0\quad(\mbox{\rm mod~} pn-1).\]
for all $n\geqslant 1$.
\item[ii.] If $cp(1-\gamma)< pk+r<c(p-r)$ and $r\leqslant \frac{p-3}2$ then there are no positive integers $k\geqslant 2 M'_p(\frac{p-(2r+1)}{p+1})$ and $c$ such that
\[{pcn\choose (pk+r)n}\equiv 0\quad(\mbox{\rm mod~} pn-1).\]
for all $n\geqslant 1$. 
\end{itemize}
\end{theorem}
\begin{proof}
\noindent{i}. Put $b=pk+r$. We have $-\frac{b}{c}\geqslant (\gamma-1)p$. Thus
\[\frac{\frac{p(c-k)}r-1}c=\frac{p(c-k)-r}{rc}=\frac{p}r-\frac{b}{rc}\geqslant\frac{p}r+\frac{\gamma p}{r}-\frac{p}{r}=1+\frac{(p-r)^2}{pr(p+1)}>1.\]
Now since $c\geqslant M_p(\frac{(p-r)^2}{pr(p+1)})$, there is a prime number $pn-1$ with
\[ c<pn-1<\frac{p(c-k)}r-1.\]
This implies the result, since $k\leqslant rn+k\leqslant c<pn-1$ and Lucas' theorem implies
\[ {pcn\choose (pk+r)n}= {c(pn-1)+c\choose k(pn-1)+rn+k}\equiv{c\choose k}{c\choose rn+k}\quad(\mbox{\rm mod~} pn-1).\]
\noindent{ii}. For $\alpha=\frac{p+1}{2(p-r)}$ we have
\[\frac{k}{\alpha k}=\frac{2(p-r)}{p+1}=1+\frac{p-(2r+1)}{p+1}>1\]
and since $\alpha k\geqslant M'_p(\frac{p-(2r+1)}{p+1})$,
there are two prime numbers $pm-1,pn-1$ with $\alpha k<pm-1<pn-1<k$. We have 
\[k=\frac{b-r}{p}<\frac{c(p-r)-r}{p}=c-\frac{r(c+1)}{p}<c.\]
Furthermore,
\[rn+k<r\cdot \frac{k+1}{p}+k=\frac{rk+b}{p}<\frac{rk+c(p-r)}{p}<\frac{rc+c(p-r)}{p}=c.\]
Moreover,
\[\frac{c}{k}<\frac{b}{k(1-\gamma)p}=\frac{kp+r}{kp\cdot\frac{(p^2+r)(p-r)}{p^2(p+1)}}\leqslant\frac{p+1}{p-r},\]
where the last inequality is true since $k\geqslant p$. We can therefore deduce that
\[pn-1< k<c<2\cdot\frac{\frac{p+1}{p-r}}{2}k=2\alpha k<2(pm-1).\]
We have $c+1\leqslant 2(pm-1)<2(pn-1)$. 
Write $c=(pn-1)+R$ and $rn+k=(pn-1)+R'$. We know that $pn-1>R>R'$. Now Lucas' theorem implies
\[{pcn\choose (pk+r)n}= {c(pn-1)+(pn-1)+R\choose k(pn-1)+(pn-1)+R'}\equiv{c+1\choose k+1}{R\choose R'}\quad(\mbox{\rm mod~} pn-1).\]
The latter is not congruent to 0, since
\[\lfloor \frac{c+1}{pn-1}\rfloor-\lfloor \frac{k+1}{pn-1}\rfloor-\lfloor \frac{c-k}{pn-1}\rfloor\leqslant 1-1-0=0.\]
\end{proof}
\begin{lemma}\label{beta} Let $p$ be an odd prime, $1\leqslant r\leqslant p-2, j=\lfloor\frac{p}{p-r}\rfloor$ and $\alpha=\frac{p}{(p-r)(j+1)}$. Then there is an $0<\varepsilon(p,r)<1$ such that
\[\alpha<\frac{p+\varepsilon(p,r)}{(p-r)(j+1)}<\frac{\frac{r}{p-r}}{j-1+\frac{r}{p}}.\]
\end{lemma}
\begin{proof}
A simple verification shows that 
\[\alpha<\frac{\frac{r}{p-r}}{j-1+\frac{r}{p}}\]
if and only if  $p-r\nmid p$ or equivalently $r\neq p-1$. This implies the existence of $\varepsilon(p,r)$.
\end{proof}
On the other hand, let $c=j(pn-1)+R, 0\leqslant R\leqslant pn-2$ and $rn+k=(pn-1)+R', 0\leqslant R'\leqslant pn-2$ and suppose $pn-1=\theta k$, where $\alpha<\theta<\beta$. Then by Lemma~\ref{beta}, 
\begin{eqnarray*}
R'+j\theta k&=& k-(pn-1)+rn+j\theta k\\
&=& k-\theta k+r\cdot\frac{\theta k+1}{p}+j\theta k\\
&=&  k(1+(-1+\frac{r}{p}+j)\theta)+\frac{r}{p}\\
&\leqslant&  k(1+(-1+\frac{r}{p}+j)\beta)+\frac{r}{p}\\
&=&  k(1+(j-1+\frac{r}{p})\beta)+\frac{r}{p}\\
&<&  k(1+\frac{r}{p-r})+\frac{r}{p}\\
&<&\frac{p}{p-r}k+\frac{r}{p-r}\\
&<&c.
\end{eqnarray*}
Hence
\[R=c-j(pn-1)=c-j\theta k> R'.\]
This shows that
\[\lfloor \frac{c}{pn-1}\rfloor-\lfloor \frac{rn+k}{pn-1}\rfloor-\lfloor \frac{c-(rn+k)}{pn-1}\rfloor= j-1-(j-1+\lfloor\frac{R-R'}{pn-1}\rfloor)=0.\]
\section{Conjecture \ref{Con3} }

In this section, using only properties of the $p$-adic valuation we give an inductive proof of Conjecture 7.3 of \cite{KRA}. For $n \in \mathbb{N}$ and a prime $p$, the $p$-adic valuation of $n$, denoted by $\nu_p(n)$ is the highest power of $p$ that divides $n$. The expansion of $n \in \mathbb{N}$ in base p is written as $n=n_0+n_1p+\ldots+n_kp^k$ with integers $0\leqslant n_i\leqslant p-1$ and $n_k\neq 0$. \textit{Legendre's classical formula} for the factorials $\nu_p(n!)=\sum_{i=1}^{\infty}\lfloor \frac{n}{p^i}\rfloor$ appears in elementary textbooks. 
\begin{theorem}
For any positive integer $m$, there are positive integers $a$ and $b$ such that $am>b$ and
\[{amn \choose bn}\equiv 0\quad (\mbox{\rm mod~} an-1)\]
for all $n\geqslant 1$.
\end{theorem}
\begin{proof}
Let $p_1=2<p_2=3<p_3=5<\ldots$ be the sequence of prime numbers. Choose $t$ such that $p_t>3m$ and put
\begin{eqnarray*}
a&=&6p_3\ldots p_t,\\
b&=&4p_3\ldots p_t.
\end{eqnarray*}
Let $n$ be a positive integer and $q^\alpha\mid an-1$ for some prime number $q$. We aim at showing that $q^\alpha\mid{amn\choose bn}$. This of course proves that $an-1\mid{amn\choose bn}$.

Write $bn$ in base $q$ in the form $\sum_{j=0}^N r_jq^j$, where $N=\alpha-1$ or $\alpha$ since $bn\geqslant q^{\alpha-1}$. At first we show that $m<r_0$. We have
\[r_0\equiv bn=4p_3\ldots p_t n\equiv 2\cdot 3^*\cdot 6p_3\ldots p_t n=2\cdot 3^* an\equiv 2\cdot 3^* \quad (\mbox{\rm mod~} q),\]
where $3^*$ is the inverse of $3$ mod $q$. We know that 
\[3^*=\left\{\begin{array}{ll}
\frac{q+1}{3} &  \mbox{if~} 3\mid q+1\\ 
\frac{2q+1}{3} & \mbox{if~} 3\mid q+2
\end{array}\right. \]
Note that $3^*$ exists since $q\neq 3$. We thus have
\[r_0=2\cdot 3^*=\left\{\begin{array}{ll}
2\cdot\frac{q+1}{3}>\frac{q}{3} &  \mbox{if~} 3\mid q+1\\ 
2\cdot\frac{2q+1}{3}-q=\frac{q+2}{3}>\frac{q}{3} & \mbox{if~} 3\mid q+2
\end{array}\right. \]
We have $\gcd(q,p_1p_2\ldots p_t)=1$. Hence $q>p_t>3m$. This shows that $m<r_0$. We therefore have
\[\lfloor\frac{bn-m}{q^i}\rfloor=\lfloor\frac{\sum_{j=0}^N r_jq^j-m}{q^i}\rfloor=\lfloor\sum_{j=i}^N r_jq^{j-i}+\frac{\sum_{j=1}^{i-1}r_jq^j+r_0-m}{q^i}\rfloor=\sum_{j=i}^N r_jq^{j-i}=\lfloor\frac{bn}{q^i}\rfloor.\]
Now let $an-1=kq^\alpha$, where $\gcd(k,q)=1$. We evaluate the $q$-adic valuation $v_q({amn\choose bn})$. If $N=\alpha$ then
\begin{eqnarray*}
v_q\big({amn\choose bn}\big)&\geqslant&\sum_{i=1}^\alpha\big(\lfloor\frac{amn}{q^i}\rfloor-\lfloor\frac{bn}{q^i	}\rfloor-\lfloor\frac{(am-b)n}{q^i}\rfloor\big)\\
&\geqslant&\sum_{i=1}^\alpha \big(mkq^{\alpha-i}-\lfloor\frac{bn}{q^i	}\rfloor-\lfloor\frac{(am-b)n}{q^i}\rfloor\big)\\
&=&\sum_{i=1}^\alpha \big(mkq^{\alpha-i}-\lfloor\frac{bn}{q^i	}\rfloor-\lfloor\frac{mkq^\alpha+m-bn}{q^i}\rfloor\big)\\
&=&\sum_{i=1}^\alpha \big(mkq^{\alpha-i}-\lfloor\frac{bn}{q^i	}\rfloor-mkq^{\alpha-i}-\lfloor\frac{m-bn}{q^i}\rfloor\big)\\
&=&\sum_{i=1}^\alpha \big(-\lfloor\frac{bn}{q^i}\rfloor-\lfloor-\frac{bn-m}{q^i}\rfloor\big)\\
&=&\sum_{i=1}^\alpha \big(-\lfloor\frac{bn}{q^i}\rfloor+\lfloor\frac{bn-m}{q^i}\rfloor+1\big)\\
&=&\sum_{i=1}^\alpha 1=\alpha,\\
\end{eqnarray*}
since $\frac{bn-m}{q^i}$ is not an integer. 

On the other hand, if $N=\alpha-1$ then 
\begin{eqnarray*}
v_q\big({amn\choose bn}\big)&=&mk+\sum_{i=1}^{\alpha-1}\big(\lfloor\frac{amn}{q^i}\rfloor-\lfloor\frac{bn}{q^i	}\rfloor-\lfloor\frac{(am-b)n}{q^i}\rfloor\big)\\
&\geqslant&mk+\alpha-1\geqslant\alpha.
\end{eqnarray*}
Thus $q^\alpha\mid{amn\choose bn}$.

\end{proof}

\begin{theorem}
For any positive integer $m$, there are positive integers $a$ and $b$ such that 
\[{amn \choose bn}\equiv 0\quad (\mbox{\rm mod~} an-a)\]
for all $n\geqslant 1$.
\end{theorem}
\begin{proof}
Let $p_1=2<p_2=3<p_3=5<\ldots$ be the sequence of prime numbers. For a positive integer $t$ put
\[b=ap_3\ldots p_t.\]
 Let $n$ be a positive integer and $q^\alpha\mid an-a$ for some prime number $q$. It is sufficient $q^\alpha\mid{amn\choose bn}$, this concludes the proof.
Let $b=r_0 + r_1q+ \ldots +r_{\alpha}q^{\alpha}$ is the $q$-adic expansions of $b$ where $0\leq r_i \leq q-1$. We have  $\gcd(q,p_1p_2\ldots p_t)=1$ and $r_0 \equiv b\quad (\mbox{\rm mod~} q)$. Now, let $an=a+kq^\alpha$ where $\gcd(k,q)=1$. We evaluate the $q$-adic valuation $v_q({amn\choose bn})$. We have 
\begin{eqnarray*}
v_q\big({amn\choose bn}\big)&=&v_q\big(\frac{(amn)!}{(bn)!(amn-bn)!}\big)\\
&=&\sum_{i=1}^\alpha\big(\lfloor\frac{amn}{q^i}\rfloor-\lfloor\frac{bn}{q^i	}\rfloor-\lfloor\frac{anm-bn}{q^i}\rfloor\big)\\
&=&\sum_{i=1}^\alpha\big(\lfloor\frac{mkq^\alpha+ma}{q^i}\rfloor-\lfloor\frac{anp_1p_2\ldots p_t}{q^i	}\rfloor-\lfloor\frac{anm-anp_1p_2\ldots p_t}{q^i}\rfloor\big)\\
&=&\sum_{i=1}^\alpha\big(\lfloor\frac{am}{q^i}\rfloor-\lfloor\frac{b}{q^i	}\rfloor-\lfloor\frac{am-b}{q^i}\rfloor\big)\\
\end{eqnarray*}
So, it is sufficient to show for any $m \in \mathbb{Z^+}$ the inequality 
\begin{eqnarray}\label{eq}
\lfloor\frac{am}{q^i}\rfloor-\lfloor\frac{b}{q^i	}\rfloor-\lfloor\frac{am-b}{q^i}\rfloor\geqslant 0
\end{eqnarray}
Let $\frac{ma}{q^\alpha}=s+r$ and $\frac{b}{q^\alpha}=s^{'}+r^{'}$ where $s, s^{'} \in \mathbb{Z^+}$ and $0\leqslant r, r^{'} \leqslant1$
, then $\lfloor\frac{am}{q^i}\rfloor-\lfloor\frac{b}{q^i}\rfloor=s+s'$ and $\lfloor\frac{am-b}{q^i}\rfloor=s+s'$ or $s+s'-1$. Therefore \ref{eq} holds and this concludes the proof.
\end{proof}
\section{Conjecture \ref{Con1}}
Maxim Vsemirnov \cite{Max} proved that the conjecture \ref{Con1} is not true for $p=5$. He also  proved the following theorem:
\begin{theorem}
Let $p=5, a=4, b=2$. If $(\alpha,\beta)\in\{(0,0), (1,0),(1,1)\}$ then
\[{4n+\alpha \choose 2n+\beta}\equiv 0,1\hspace{0.2cm} \text{or}\hspace{0.2cm} 4\quad (\mbox{\rm mod~} 5);\]
If $(\alpha,\beta)\in\{(2,1), (3,1),(3,2)\}$ then
\[{4n+\alpha \choose 2n+\beta}\equiv 0,2\hspace{0.2cm} \text{or}\hspace{0.2cm} 3\quad (\mbox{\rm mod~} 5);\]
\end{theorem}
In the following, we give a proof for a special case of Conjecture \ref{Con1}
We know that if $\gcd(x,y)=1$ then there is an integer $1\leqslant x'\leqslant y-1$ such that $y\mid xx'-1$. We denote this $x'$ by $\mbox{Inv}_y(x)$. Moreover, for an integer $x$ we denote the $p$-adic valuation of $x$ by $v_p(x)$.
\begin{theorem}
Let $a$ and $b$ be positive integers with $a>b$, let $\alpha$ and $\beta$ be integers and let $d=\gcd(a,b), c=\frac{a}{d}, e=\gcd(p-1,a)$. Furthermore, let $p$ be a prime such that $p>a+2b$. Then 
\begin{itemize}
\item[i.] if $e< c$ or $v_2(a)\leqslant v_2(p-1)$ then for each $r=0,1,\ldots,p-1$, there are infinitely many positive integers $n$ such that 
\[{an+\alpha \choose bn+\beta}\equiv r\quad (\mbox{\rm mod~} p);\]
\item[ii.] if $e\geqslant c$ and $v_2(a)>v_2(p-1)$ then for each 
\[r\notin\{(2\mu-1)e+p+\alpha-2+r': 0\leqslant r'\leqslant e-c, \lceil\frac{e+2-p-\alpha}{2e}\rceil\leqslant\mu\leqslant\lceil\frac{c+1-\alpha}{2e}\rceil\},\]
there are infinitely many positive integers $n$ such that 
\[{an+\alpha \choose bn+\beta}\equiv r\quad (\mbox{\rm mod~} p).\]
\end{itemize}
\end{theorem}
\begin{proof}
By Euler's totient theorem, we have $p^{\varphi(a)}\equiv1~ (\mbox{\rm mod~} a)$, since $\gcd(p,a)=1$. For an arbitrary positive integer $N$, put $u=N\varphi(a)$. Thus 
\[p^{ui}\equiv1~(\mbox{\rm mod~} a), \quad i\in\mathbb{N}.\]
In particular, there is an $m$ such that $p^u-1=am$. Thus $m=-\mbox{Inv}_p(a)$. Put $t=(p-1)-r$. Write $c-t-\alpha=\mu e+\rho$, where $0\leqslant \rho\leqslant e-1$. 
Note that $e\mid c-t-\alpha-\rho$. 
Suppose 
\[\varepsilon=\left\{\begin{array}{ll}
0 & \mbox{if~} \rho\leqslant c-2\\
1 & \mbox{otherwise~} 
\end{array}\right. \]
If $e<c$ then put
\begin{eqnarray*}
K&=&\frac{c-t-\alpha-\rho}{e}\cdot\big(p\mbox{Inv}_{\frac{a}{e}}(\frac{p(p-1)}e)\big)\\&&+(c-t-\alpha-\rho)am\mbox{Inv}_p(a+1)-(\beta-1)a^2m\mbox{Inv}_p(b(a+1))+Lpa,
\end{eqnarray*}
where $L$ is sufficiently large so that $K>1$. Note that $\varepsilon=0$ in this case, since $\rho\leqslant e-1\leqslant c-2$.

If $e\geqslant c$ and $v_2(a)\leqslant v_2(p-1)$ then put
\begin{eqnarray*}
K&=&\frac{c-t-\alpha-\rho}{e}\cdot\big(p\mbox{Inv}_{\frac{a}{e}}(\frac{p(p-1)(1+\varepsilon)}e)\big)\\&&+(c-t-\alpha-\rho)am\mbox{Inv}_p(a+1)-(\beta-1)a^2m\mbox{Inv}_p(b(a+1))+Lpa,
\end{eqnarray*}
where $L$ is sufficiently large so that $K>1$. Note that $\mbox{Inv}_{\frac{a}{e}}(1+\varepsilon)$ exists, since $\frac{a}{e}$ is odd in this case. 

Finally, if $e\geqslant c$ and $v_2(a)> v_2(p-1)$ then put
\begin{eqnarray*}
K&=&\frac{c-t-\alpha-\rho}{(1+\varepsilon)e}\cdot\big(p\mbox{Inv}_{\frac{a}{(1+\varepsilon)e}}(\frac{p(p-1)}e)\big)\\&&+(c-t-\alpha-\rho)am\mbox{Inv}_p(a+1)-(\beta-1)a^2m\mbox{Inv}_p(b(a+1))+Lpa,
\end{eqnarray*}
where $L$ is sufficiently large so that $K>1$. Note that $\frac{c-t-\alpha-\rho}{e}$ is even by our assumption on $r$ in this case. 

In each of the above cases we have
\begin{eqnarray*}
K(p-1)(1+\varepsilon)&\equiv& c-t-\alpha-\rho~(\mbox{mod~}a),\\
mb\big(K(p-1)(a+1+\varepsilon)-(c-t-\alpha-\rho)\big)&\equiv& \beta-1~(\mbox{mod~}p).
\end{eqnarray*}
Now choose
\[M=K(p-1)(d(c-1)+1)-(c-1)+\rho,\]
and
\begin{eqnarray*}
\mathbb{I}_2&=&\{M-k(c-1): k=0,1,\ldots, K(p-1)(d+\varepsilon)-1\},\\
 \mathbb{I}_1&=&\{1,2,\ldots,M\}\setminus\mathbb{I}_2.
\end{eqnarray*}
We have
\begin{eqnarray*}
&&p^{u(M+1)}-t-\sum_{i\in\mathbb{I}_1}p^{ui}-\sum_{i\in\mathbb{I}_2}2p^{ui}-\alpha\\
&\equiv& 1-t-(M-K(p-1)(d+\varepsilon))-2K(p-1)(d+\varepsilon)-\alpha\\
&=&1-t-\alpha-K(p-1)(a+1+\varepsilon)+(c-1)-\rho\\
&\equiv& c-t-\alpha-\rho-K(p-1)(1+\varepsilon)\\
&\equiv& 0 \quad(\mbox{mod~}a).
\end{eqnarray*}
Hence, there is a positive integer $n$ such that 
\[an+\alpha=p^{u(M+1)}-t-\sum_{i\in\mathbb{I}_1}p^{ui}-\sum_{i\in\mathbb{I}_2}2p^{ui}.\]
Write $an+\alpha$ in base $p$ as the form $\sum_{s=0}^{u(M+1)}a_sp^s$.  Then we have
\[a_s=\left\{\begin{array}{ll}
p-1-t & \mbox{if~} s=0\\
p-2 &  \mbox{if~} s=ui \mbox{~for some~}i\in\mathbb{I}_1\\ 
p-3 &  \mbox{if~} s=ui \mbox{~for some~}i\in\mathbb{I}_2\\ 
p-1 & \mbox{otherwise~} 
\end{array}\right. \]
We now aim at finding digits of $bn+\beta$ in base $p$. If $bn+\beta=\sum_{s=0}^{u(M+1)}b_sp^s$ then 
$b_s$ is the remainder of $\lfloor\frac{bn+\beta}{p^s}\rfloor$ mod $p$. In fact, we need to find $b_s$ for $s=0,u,2u,\ldots, Mu$.

We now have
\begin{eqnarray*}
bn+\beta&=&\frac{b}{a}\big(p^{u(M+1)}-t-\sum_{i\in\mathbb{I}_1}p^{ui}-
\sum_{i\in\mathbb{I}_2}2p^{ui}\big)-\frac{b}{a}\alpha+\beta \\
&=&\frac{b}{a}\big(p^{u(M+1)}-1-\sum_{i\in\mathbb{I}_1}(p^{ui}-1)-\sum_{i\in\mathbb{I}_2}2(p^{ui}-1)\big)\\
&&+\frac{b}{a}(1-t-(M-K(p-1)(d+\varepsilon))-2K(p-1)(d+\varepsilon)-\alpha)+\beta\\
&=&\frac{b}{a}\big(p^{u(M+1)}-1-\sum_{i\in\mathbb{I}_1}(p^{ui}-1)-\sum_{i\in\mathbb{I}_2}2(p^{ui}-1)\big)\\
&&+\beta-\frac{b}{a}\big(c-t-\alpha-\rho-K(p-1)(a+1+\varepsilon)\big).\\
\end{eqnarray*}
Thus
\begin{eqnarray*}
bn+\beta&\equiv&-mb\big(p^{u(M+1)}-1-\sum_{i\in\mathbb{I}_1}(p^{ui}-1)-\sum_{i\in\mathbb{I}_2}2(p^{ui}-1)\big)\\
&&+\beta-mb\big(K(p-1)(a+1+\varepsilon)-(c-t-\alpha-\rho)\big)\\
&\equiv&\beta-(\beta-1)\equiv 1~(\mbox{mod~}p).
\end{eqnarray*}
This shows that $b_0=1$. Given $s$, for $j=1,2$ let $I_{s,j}$ be the number of $i\in\mathbb{I}_j$ with $i\geqslant s$.
For $s\in\mathbb{I}_j$ we have
\begin{eqnarray*}
\lfloor\frac{bn+\beta}{p^{us}}\rfloor&=&\lfloor\frac{b}{a}\big(p^{u(M+1-s)}-1-\sum_{s\leqslant i\in\mathbb{I}_1}(p^{u(i-s)}-1)-\sum_{s\leqslant i\in\mathbb{I}_2}(2p^{u(i-s)}-2)\\
&&-\sum_{s>i\in\mathbb{I}_1}\frac{1}{p^{u(s-i)}}-\sum_{s>i\in\mathbb{I}_2}\frac{2}{p^{u(s-i)}}-\frac{t}{p^{us}}-I_{s,1}-2I_{s,2}+1\big)+\frac{\beta}{p^{us}}\rfloor\\
&=&\frac{b}{a}\big(p^{u(M+1-s)}-1-\sum_{s\leqslant i\in\mathbb{I}_1}(p^{u(s-i)}-1)-\sum_{s\leqslant i\in\mathbb{I}_2}(2p^{u(s-i)}-2)\big)\\
&&+\lfloor\frac{b}{a}\big(-\sum_{s>i\in\mathbb{I}_1}\frac{1}{p^{u(s-i)}}-\sum_{s>i\in\mathbb{I}_2}\frac{2}{p^{u(s-i)}}-\frac{t}{p^{us}}-I_{s,1}-2I_{s,2}+1\big)+\frac{\beta}{p^{us}}\rfloor\\
&\equiv& -mb(-1+I_{s,1}+2I_{s,2}-j)-\lfloor\frac{b}{a}(-1+I_{s,1}+2I_{s,2})\rfloor-1\quad(\mbox{mod~} p).
\end{eqnarray*}
Let $I_{s,1}+I_{s,2}-1=cq_s+r_s$, where $0\leqslant r_s<c$. Then for $s\in\mathbb{I}_j$ we have
\[b_{su}=mb(j-r_s)-\lfloor\frac{br_s}{a}\rfloor-1\equiv m(a\lfloor\frac{br_s}{a}\rfloor+a-b(r_s-j))\quad(\mbox{mod~}p).\]
Let us evaluate $r_s$ for $s\in\mathbb{I}_j$. If $j=2$ then $s=M-k(c-1)$ for some $k=0,1,\ldots, K(p-1)(d+\varepsilon)-1$. Thus
\[I_{s,1}=M-(M-k(c-1))+1-(k+1),\quad I_{s,2}=k+1.\]
So 
\[cq_s+r_s=k(c-1)+1+(k+1)-1\equiv1\quad(\mbox{mod~}c).\]
Hence $r_s=1$, whenever $s\in\mathbb{I}_2$. Note that we have $K(p-1)(d+\varepsilon)$ times occurrence of $r_s=1$.

Moreover, if $j=1$ then $s=M-k(c-1)-s'$ for some $k=0,1,\ldots,K(p-1)(d+\varepsilon)-1$ and $s'=1,2,\ldots,c-2$. Thus
\[I_{s,1}=M-(M-k(c-1)-s')+1-(k+1),\quad I_{s,2}=k+1.\]
So 
\begin{eqnarray*}
cq_s+r_s&=&k(c-1)+s'+1+(k+1)-1\equiv s'+1\quad(\mbox{mod~}c).
\end{eqnarray*}
Hence $r_s=2,\ldots,c-1$, whenever $s\in\mathbb{I}_1$. Note that we have $K(p-1)(d+\varepsilon-1)$ times occurrence of $r_s=\rho+2-\varepsilon(c-1),\ldots,c-1$ and $K(p-1)(d+\varepsilon)$ times occurrence of $r_s=2,\ldots,\rho+1-\varepsilon(c-1)$.

Now we show that if $s\in\mathbb{I}_1$ then $b_{su}+1\not\equiv 0~(\mbox{mod~}p)$ and if $s\in\mathbb{I}_2$ then $b_{su}+1,b_{su}+2\not\equiv 0~(\mbox{mod~}p)$.

Let $s\in\mathbb{I}_j$. Then
\[b_{su}+1\equiv m(a\lfloor\frac{br_s}{a}\rfloor-b(r_s-j))\quad(\mbox{mod~}p).\]
Now if $p\mid b_{su}+1$ then $p\mid a\lfloor\frac{br_s}{a}\rfloor-b(r_s-j)$. The latter holds if and only if $a\lfloor\frac{br_s}{a}\rfloor-b(r_s-j)=0$, since  
\[|b(r_s-j)-a\lfloor\frac{br_s}{a}\rfloor|\leqslant a(\frac{br_s}{a}-\lfloor\frac{br_s}{a}\rfloor)+jb\leqslant a+2b<p\]
Thus we should have $a\mid b(r_s-j)$ which implies that $c\mid r_s-j$. This is a contradiction, since $r_s<c$ and $r_s\neq j$ whenever $s\in\mathbb{I}_j$. 

Let $s\in\mathbb{I}_2$. Then
\[b_{su}+2\equiv m(a\lfloor\frac{br_s}{a}\rfloor-a-b(r_s-2))\quad(\mbox{mod~}p).\]
We know that $r_s=1$ whenever $s\in\mathbb{I}_2$. Thus if $p\mid b_{su}+2$ then we should have $p\mid a-b$. The latter is impossible since $p>a-b$.

We therefore have
\begin{eqnarray*}
{an+\alpha \choose bn+\beta}&\equiv&{p-1-t\choose b_0}\prod_{s\in\mathbb{I}_1}(b_{su}+1)\prod_{s\in\mathbb{I}_2}(b_{su}+1)(b_{su}+2)\\
&\equiv&-(1+t)\prod_{r_s=2}^{\rho+1-\varepsilon(c-1)}(b_{su}+1)^{K(p-1)(d+\varepsilon)}\prod_{r_s=\rho+2-\varepsilon(c-1)}^{c-1}(b_{su}+1)^{K(p-1)(d+\varepsilon-1)}\\
&&\cdot(mb)^{K(p-1)(d+\varepsilon)}(m(b-a))^{K(p-1)(d+\varepsilon)}\\
&\equiv&-(1+t)\\
&\equiv&-(1+(p-1)-r)\\
&\equiv& r\quad(\mbox{mod~}p).
\end{eqnarray*}
Note that there are infinitely many such $n$, since $N$ was arbitrary.
\end{proof}
\section{acknowledgments}
The authors would like to thank the anonymous referee for her/his careful reading and invaluable suggestions. This paper was written during the first author sabbatical period in the Department of Mathematics at the Vienna University. The authors wish to special
thanks to professor Christian Krattentheler and M. Farrokhi D. G for their helpful advices.


\begin{thebibliography}{99}
\bibitem{Bab} C. Babbage,\textit{ Demonstration of a theorem relating to prime numbers}, Edinburgh Philosophical J. 1 (1819), 46–49 .
\bibitem{KRA}V. J. W. Guo and C. Krattenthaler, \textit{Some divisibility properties of binomial and $q$-binomial coefficients}, Journal of Number Theory, {\bf 135},  (2014), 167-189.
\bibitem{Har} G. H. Hardy and E. M. Wright, \textit{An Introduction to the Theory of Numbers}, 6th ed., Oxford University Press, 2008, p. 494. .
\bibitem{Me} R. M\v{e}strovi\'{c}, {\em Wolstenholme's theorem: its generalizations and extensions in the last hundred and fifty years (1862--2012)}, preprint, \url{arxiv:1111.3057}.
\bibitem{Sun3}Z.-W. Sun, {\em On divisibility of binomial coefficients}, J. Austral. Math. Soc. {\bf 93} (2012), 189--201
\bibitem{Max} M. Vsemirnov, {\em On a conjecture of Guo and Krattenthaler}, International Journal of number theory . {\bf 10}, No. 6 (2014), 1541-1543.
\bibitem{Wolstenholme}J. Wolstenholme, {\em On certain properties of prime numbers}, Quart. J. Pure Appl. Math. {\bf 5} (1862), 35--39.
\end{thebibliography}
\end{document}